\documentclass[12pt,leqno,twoside]{amsart}
\usepackage{amssymb,amsfonts,amsmath,amsthm, amstext}
\usepackage{latexsym}
\usepackage{verbatim}
\usepackage{graphicx}
\usepackage{subcaption}
\usepackage{pgf,tikz}
\usepackage{float}

\usepackage[latin1]{inputenc}
\usepackage[T1]{fontenc}
\usepackage[colorlinks=true, pdfstartview=FitV, linkcolor=blue, citecolor=blue, urlcolor=blue]{hyperref}
\usepackage{amstext,amsmath,amscd, bezier,indentfirst,amsthm,amsgen,enumerate}
\usepackage[all,knot,arc,import,poly]{xy}
\usepackage{amsfonts,xcolor, soul}  
\usepackage{graphicx}
\usepackage{srcltx}
\usepackage{enumitem}
\usepackage{epsfig}

\usepackage{tikz}
\usepackage{mathrsfs}
\usepackage[margin=1in]{geometry}
\usepackage[normalem]{ulem}

\newtheorem{theorem}{Theorem}[section]
\newtheorem{corollary}[theorem]{Corollary}
\newtheorem{lemma}[theorem]{Lemma}

\theoremstyle{definition}
\newtheorem{definition}[theorem]{\sc Definition}
\theoremstyle{remark}

\theoremstyle{remark}

\theoremstyle{remark}
\newtheorem{example}[theorem]{\sc Example}
\theoremstyle{remark}

\theoremstyle{remark}

\newtheorem{lem/defi}[thm]{Lemma/Definition}
\newtheorem{defi/lem}[thm]{Definition/Lemma}
\newtheorem{prop/defi}[thm]{Proposition/Definition}

\theoremstyle{definition}

\theoremstyle{remark}

\renewcommand{\Box}{\square}    

\newcommand{\hot}{\mathrm{h.o.t.}}

\newcommand{\im}{\mathop{\rm{Im}}\nolimits}

\newcommand{\ord}{\mathop{\rm{ord}}}

\newcommand{\grad}{\mathop{\rm{grad}}\nolimits}

\newcommand{\Horn}{{\rm{Horn}}}

\newcommand{\fin}{\hspace*{\fill}$\Box$}

\newcommand{\GC}{\mathcal{GC}}
\newcommand{\mcd}{\mathcal{D}}

\newcommand{\Cone}{\mathrm{Cone}}

\newcommand{\be}{\begin{equation}}
\newcommand{\ee}{\end{equation}}

\newcommand{\beqn}{\begin{eqnarray*}}
\newcommand{\eeqn}{\end{eqnarray*}}

\newcommand{\cC}{{\mathcal C}}

\newcommand{\bR}{{\mathbb R}}
\newcommand{\bC}{{\mathbb C}}

\newcommand{\F}{\mathbb{F}}

\newcommand{\bN}{{\mathbb N}}

\newcommand{\bQ}{{\mathbb Q}}

\title[On higher Lipschitz invariants]{On higher Lipschitz invariants}

\author{Piotr Migus}
\address{Air Force Institute of Technology,
ul. Ksi\c{e}cia Boles\l awa 6,
01-494 Warsaw, Poland}
\email{migus.piotr@gmail.com}

\author{Lauren\c tiu P\u aunescu}
\address{School of Mathematics and Statistics, University of Sydney,
  Sydney, NSW, 2006, Australia.}
\email{laurent@maths.usyd.edu.au}

\author{Mihai Tib\u ar}
\address{Math\' ematiques, UMR 8524 CNRS,
Universit\'e de Lille, \  59655 Villeneuve d'Ascq, France.}
\email{mtibar@univ-lille.fr}


\begin{document}

\begin{abstract} 
 We find new  bi-Lipschitz invariants for functions of two complex variables.
  \end{abstract}

\maketitle

\section{Introduction and preliminaries}\label{s:intro}
   
While the bi-Lipschitz
equivalence of complex analytic set germs does not admit moduli,  Henry and Parusi\'nski \cite{HP} firstly showed that, despite  the tempting belief that the same holds true for function germs,   moduli exist in this case.  See also the references to moduli in the related studies \cite{FR} and  \cite{CR}.

Based on the clustering gradient canyons, we have shown in \cite{PT} that the multi-level ``identity card'' of a 2-variable holomorphic function germ  $f: (\bC^2, 0) \to (\bC, 0)$
is a discrete bi-Lipschitz invariant. We have pointed out in \cite{MPT} how the modular bi-Lipschitz invariant found by Henry and Parusi\'nski \cite{HP}
 can be embedded into this picture as the first of the coefficients in the expansion of $f$ along polars.

\

In this paper we infer a new and different spirit in the research of invariants. Let $f= g\circ \varphi$ for $f,g$ holomoprhic function germs of 2 variables, and some bi-Lipschitz homeomorphism $\varphi$. Our main result, Theorem \ref{t:coeff},
shows that, surprisingly,  one may determine a development of $\varphi$ to a certain extent. The power of this result is illustrated by our second example in  \S\ref{s:examples}.

As a byproduct of our refined study \cite{PT}, we find here new bi-Lipschitz invariants within the identity card of the clusters, toward a complete list of bi-Lipschitz invariants enabling us to decide  the bi-Lipschitz type of a given
two-variables holomorphic function germ. This contributes to a long standing research project, as one can move away from a merely homeomorphism equivalence  (too weak) but not too close to a diffeomorphism equivalence (too rigid).
Our  Corollary \ref{c:invar1} unveils  a second level bi-Lipschitz invariant of Henry-Parusi\'nski type, based on two gradient canyons and their contact order. By Corollary \ref{c:invar2} we show how  to build, in a similar way, a third level bi-Lipschitz invariant of HP-type in case one disposes of 3 or more disjoint gradient canyons, and their contact orders satisfy certain conditions.  As expected, this procedure may continue recursively toward still higher level, producing finitely many invariants of HP-type, cf Corollary \ref{c:invar2}.

Our two examples have a key role in understanding two types of applications of Corollary  \ref{c:invar1}. Example \ref{example1} shows  how this result decides that any two generic function germs from a certain family are not bi-Lipschitz equivalent.  In contrast, Corollary  \ref{c:invar1} is no longer conclusive in Example \ref{example2}, where two generic function germs in a certain family turn out not to be bi-Lipschitz equivalent. To reach this conclusion we need to employ our key Theorem \ref{t:coeff}, which yields the existence of new bi-Lipschitz moduli at this level.

\ 

\noindent \emph{Added in proof}. After uploading our manuscript on arXiv,  N. Nguyen pointed out to us his newly posted arxiv preprint \cite{Ng} in which he observes a version of our Corollary  \ref{c:invar1} by deriving it directly from our constructions in \cite{PT}.

\section{Preliminaries}\label{s:prelim}

 Let $f,g:(\bC^2,0)\to (\bC,0)$ be analytic function germs such that $f=g\circ \varphi$, where $\varphi:(\bC^2,0)\to (\bC^2,0)$ is a bi-Lipschitz homeomorphism. 

Consider a holomorphic map germ 
$$
\alpha :(\bC,0)\to(\bC^2,0), \quad \alpha(t)=(z(t),w(t)) \not \equiv 0.
$$
The image set germ $\alpha_*=\im(\alpha)$ is a \emph{curve germ} at $0\in \bC^2$, also called a \emph{holomorphic arc} at $0$. There is a well-defined tangent line $T(\alpha_*)$ at $0$,   $T(\alpha_*)\in\bC P^1$. 

In order to introduce our results, we need to recall here several notations and facts from \cite{PT} and \cite{MPT}.

\medskip

Let $\F$ be the field of convergent fractional power series in an indeterminate $y$. By the Newton-Puiseux Theorem we have that $\F$ is algebraically closed, cf \cite{BK}, \cite{Wa}.

A non-zero element of $\F$ has the form 
\begin{equation}\label{puis}
\eta(y)=a_0y^{n_0/N}+a_1y^{n_1/N}+a_2y^{n_2/N}+\cdots,\quad n_0<n_1<n_2<\cdots,
\end{equation}
where $a_i \in \bC^*$ and  $N,n_i\in \bN$ with $\gcd(N,n_0,n_1,\dots)=1$, $\lim \sup |a_i|^{\frac{1}{n_i}}<\infty$. The elements of $\F$ are called \emph{Puiseux arcs}. There are $N-1$ \emph{conjugates} of $\eta,$ which are the Puiseux arcs of the form
$$
\eta^{(k)}_{conj}(y):=\sum a_i\varepsilon^{kn_i}y^{n_i/N}, \quad \varepsilon:=e^{\frac{2\pi\sqrt{-1}}{N}},
$$
where $k\in\{0,\dots ,N-1\}$.

By the \emph{order of a Puiseux arc} \eqref{puis} we mean $\ord \eta(y):=\frac{n_0}{N}$, and by the \emph{Puiseux multiplicity} we mean $m_{puiseux}(\eta)=N$, cf \cite{BK}, \cite{Wa}. 

Let $\F_1:=\{\eta \in \F \mid \ord \eta (y)\geq 1\}$.   
For any $\eta \in \F_1$ with $\ord \eta (y)\geq 1$, the following map germ:
$$
\eta_{par} :(\bC,0)\to(\bC^2,0), \quad t\mapsto (\eta(t^N),t^N), \quad N:=m_{puiseux}(\eta),
$$
is holomorphic, and all the conjugates of $\eta$ lead to the same irreducible curve $\im\eta_{par}$, which will be denoted by $\eta_*$.

\begin{definition}[Contact order of holomorphic arcs, cf \cite{BK}, \cite{Wa} etc]\ \\
 The \emph{contact order} between two different holomorphic arcs $\alpha_*$ and $\beta_*$, where $\alpha, \beta \in \F_1$, is defined as: 
 $$
 \max \ord_{y} (\alpha'(y) - \beta'(y))
 $$
where the maximum is taken over all conjugates $\alpha'$ of $\alpha$ and $\beta'$ of $\beta$.
\end{definition}

%
\medskip

Let $f:(\bC^2,0)\rightarrow(\bC,0)$ be a holomorphic function germ, and let $m:= \ord_{0}f$. We say that $f$ is \emph{mini-regular in $x$ of order $m$},  if the initial form of the Taylor expansion of $f$ is not equal to $0$ at the point $(1,0)$, in other words $f_m(1,0)\neq 0$ where $f(x,y)=f_{m}(x,y)+f_{m+1}(x,y)+\hot$ is the homogeneous Taylor expansion of $f$.

Let $f:(\bC^2,0)\to (\bC,0)$ be a mini-regular holomorphic function in $x$. We have the following Puiseux factorisations of $f$, and of its derivative $f_x$ with respect to $x$:
$$
f(x,y)=u\cdot \prod_{i=1}^m(x-\zeta_i(y)), \quad f_x(x,y)=v\cdot \prod_{i=1}^{m-1}(x-\gamma_i(y)),
$$ 
where $u,v$ are units. \\ 
 If $f_x=g_1^{q_1} \cdots g_p^{q_p}$ is the decomposition into irreducible factors, then $\Gamma_i=\{g_i=0\}, i=1,\dots, p,$ if $g_i$ is not a factor of $f$ as well;  there exists at least one Puiseux root $\gamma$ of $f_x$ sudch that $g_i(\gamma(y),y)\equiv 0$. 

\medskip

\subsection{Gradient degree and gradient canyon.}\label{ss:degree}
Let $\gamma$ be a polar arc of $f$, thus such that $f(\gamma(y),y)\not\equiv 0$ and $f_x(\gamma(y),y) \equiv 0$. The \emph{gradient degree} $d(\gamma)$ is the smallest number $q$ such that 
$$
\ord_y(||\grad f(\gamma(y),y)||)=\ord_y(||\grad f(\gamma(y)+uy^q,y)||),
$$
holds for generic $u\in \bC$. The \emph{gradient canyon} $\GC(\gamma_*)$ is the subset of all curve germs $\alpha_*$, where $\alpha$ is a Puiseux arc of the form 
$$
\alpha(y):=\gamma +uy^{d(\gamma)}+\hot
$$
for any $u\in \bC$. It follows that all the polars in the same canyon have the same gradient degree, and therefore this is called the  \emph{canyon degree}. 

\medskip

\subsection{The order $\ord_y f(\gamma(y),y)$.}\label{ss:order-h}
   The order $h:=\ord_y f(\gamma(y),y)$, where $\gamma$ is a polar arc of $f$, is by definition a positive rational number. To each such order $h$, one associates at least one bar $B(h)$ in the Kuo-Lu tree of $f$, cf \cite{kuo-lu}, such that, if $\gamma$ grows from $B(h)$, then $h=\ord_y f(\gamma(y),y)$. The order $h$ is constant in the same canyon $\cC$,  see e.g. \cite{PT}, thus we may denote it by $h_{\cC}$.

\medskip

\subsection{Clusters of canyons.}\label{ss:clusters}
Let $G_\ell(f)$ be the subset of canyons tangent to the line $\ell$ of the tangent cone $\Cone_0(f)$ at the origin of the curve $Z(f)$. Let then $G_{\ell,d,B(h)}(f)\subset G_\ell(f)$ denote the union of gradient canyons of a fixed degree $d>1$, the  polars of which grow on the same bar $B(h)$.

\medskip

\subsection{Contact of canyons.}\label{ss:contact}

A fixed gradient canyon $\GC_i(f) \in G_{\ell,d,B(h)}(f)$ has a well-defined order of contact $k(i,j)$ with some other gradient canyon $\GC_j(f) \in G_{\ell, d,B(h)}(f)$ of the same cluster. This is, by definition, the contact order between some polar in $\GC_i(f)$ and some polar in $\GC_j(f)$.

The number $k(i,j)$ counts also the multiplicity of each such contact, that is, the number of canyons $\GC_j(f)$ of the cluster $G_{\ell,d,B(h)}(f)$ which have exactly the same contact with $\GC_i(f)$.

\medskip

\subsection{Sub-clusters of canyons.}\label{ss:subclusters}
Let $K_{\ell,d,B(h),i}(f)$ denote the (un-ordered) set of those contact orders $k(i,j)$ of the fixed canyon $\GC_i(f)$, counted with multiplicity.

Finally, let $G_{\ell,d,B(h),\omega}(f)$ be the union of canyons from $G_{\ell,d,B(h)}$ which have exactly the same set $\omega=K_{\ell,d,B(h),i}(f)$ of  orders of contact $\ge 0$ with the other canyons from $G_{\ell,d,B(h)}$. We then have a partition:
$$
G_{\ell,d,B(h)}(f)= \bigsqcup_{\omega} G_{\ell,d,B(h),\omega}(f).
$$

\medskip

\subsection{Lipschitz invariants and ``identity card''.}\label{ss:lipinvar}
By \cite{PT} and \cite{MPT}, for any degree $d>1$, any bar $B(h)$,  and any rational $h$, the following are bi-Lipschitz invariants:
\begin{enumerate}
\item the cluster of canyons $G_{\ell,d,B(h)}$,
\item the set of contact orders $K_{\ell,d,B(h),i}(f)$, and for each such set, the sub-cluster of canyons $G_{\ell,d,B(h),K_{\ell,d,B(h),i}}(f)$,
\item the bi-Lipschitz homeomorphism preserves the contact orders between any two clusters of type $G_{\ell,d,B(h),K_{\ell,d,B(h),i}}(f)$.
\end{enumerate} 

\medskip

On the other hand, let $\cC \in G_\ell(f)$ be a canyon, and let $\gamma$ be some polar arc in $\cC$, where $\ell$ is in the tangent cone of $Z(f)$. If $d_{\cC}$ denotes the gradient degree of $\cC$, then
$$
f(\gamma(y),y)=a_{h_{\cC}}y^{h_\cC}+\hot
$$
by \cite[Proposition 3.7.]{PT},  where $a_{h_{\cC}}$ and  $h_{\cC}$ depend only on the canyon. 

The main theorem of \cite{HP}, completed by \cite{PT},  tells that we also have the following bi-Lipschitz invariant:

\begin{enumerate}
\item[(d)] The effect of the bi-Lipschitz map $\varphi$ on each such couple $(d_{\cC}, a_{h_{\cC}})$ is the identity on $d_{\cC}$, and the multiplication of $a_{h_{\cC}}$ by $c^{- h_{\cC}}$, where $c$ is a certain non-zero constant which is the same for all canyons $\cC \in G_\ell(f)$.
\end{enumerate}

We say that the above bunch of data, which yield the described bi-Lipschitz invariants, constitute the ``identity card'' of $f$.  

\bigskip

\noindent \textbf{Acknowledgements}.
The authors acknowledge partial support from the grant ``Singularities and Applications'' - CF 132/31.07.2023 funded by the European Union - NextGenerationEU - through Romania's National Recovery and Resilience Plan. 
\emph{L.P.} and \emph{M.T.} acknowledge support by the PHC FASIC program, project no. 51448UE, funded by the French Ministry for Europe and Foreign Affairs, the French Ministry for Higher Education and Research and the Australian academic and institutional partners. 
\emph{M.T.} acknowledges support of the Laboratoire Painlevé and the R-CDP-24-004-C2EMPI project.

\section{Construction of bi-Lipschitz invariants}
In the following, we will assume that $f,g:(\bC^2,0)\to (\bC,0)$ are analytic function germs such that $f=g\circ \varphi$, where $\varphi=(\varphi_1,\varphi_2):(\bC^2,0)\to (\bC^2,0)$ is a bi-Lipschitz homeomorphism and $f$ is mini-regular in $x$.

\medskip

\noindent \textbf{Notations.} Let $F$ and $G$ be functions (possible multivaluate) of the variable $x\in \bR^n$. 

We say that $|F(x)|\lesssim |G(x)|$ or that \emph{``$F$ is $O(G)$''}, if there exists a constant $K>0$ such that the inequality $|F(x)|\leq K |G(x)|$ holds in some neighbourhood of the origin. 

We write:

$\bullet$  ``$|F|\sim |G|$'' if $|F(x)|\lesssim |G(x)|$ and $|G(x)|\lesssim |F(x)|$. 


$\bullet$  ``$F\ll G$'', and we say that \emph{``$F$ is $o(G)$''} if $\frac{|F(x)|}{|G(x)|}\to 0$, whenever $x\to 0$.

\subsection{Local variables}\label{ss:locvar0}\

\begin{lemma}\label{l:yY}
Let $\gamma$ be a polar arc of $f$ tangent to the line $\ell  \in \Cone_0(f)$, and let $\GC(\gamma_{*})\in G_\ell(f)$ be its canyon. There is a polar $\gamma'$ of $g$ such that its canyon $\GC(\gamma'_*)$ has the same degree $d$ as $\GC(\gamma_{*})$, and such that:
$$
\varphi(\gamma(y),y)=\bigl(\varphi_{1}(\gamma(y),y),\varphi_2(\gamma(y),y)\bigr)=\bigl(\gamma'(Y)+O( Y^{d}), Y\bigr),
$$
where $Y:= \varphi_2(\gamma(y),y)$ is a local variable.
\end{lemma}

Before giving the proof, we need to recall some more notations from \cite{PT}. 

\medskip

\noindent \textbf{Horn domains.}
For some fixed $\alpha_*$, where $\alpha$ is a Puiseux series as before, one defines the infinitesimal disks: 
$$
\mcd^{(e)}(\alpha_*;\rho):=\{\beta_*\mid \beta(y)=[J^{e}(\alpha)(y)+cy^e]+\hot,|c|\leq \rho\},
$$
where $1\leq e<\infty$, $\rho\geq 0$.

Consider $\mcd^{(e)}(\alpha_*;\rho)$ of finite order $e\geq 1$ and finite radius $\rho>0$, and a compact  ball $B(0;\eta):=\{(x,y)\in \bC^2 \mid \sqrt{|x|^2+|y|^2}\leq \eta\}$ with small enough $\eta >0$.  
Let then
$$
\Horn^{(e)}(\alpha_*;\rho;\eta):=\bigl\{ (x,y)\in B(0;\eta)\mid x=\beta(y)=J^e(\alpha)(y)+cy^e, |c|\leq \rho\bigr\}
$$
be the \emph{horn domain} associated to $\mcd^{(e)}(\alpha_*;\rho)$. We note that  $(x,y)\in \Horn^{(e)}(\alpha_*;\rho;\eta)$ if and only if $|x-J^e(\alpha (y))|\leq \rho |y^e|$, in particular  $x=\alpha (y) + O(y^e)$.

\medskip

\noindent \textbf{Canyon disks within the Milnor fibre.}
Let $\mcd^{(e)}_{\gamma_*,\varepsilon}(\lambda;\eta)$ be the union of disks\footnote{By \emph{disk} we mean \emph{homeomorphic to an open disk}.} in the Milnor fibre $\{f=\lambda\}\cap B(0;\eta)$ of $f$ defined as follows  
$$
\mcd^{(e)}_{\gamma_*,\varepsilon}(\lambda;\eta):=\{f=\lambda\} \cap \Horn^{(e)}(\gamma_*;\varepsilon;\eta),
$$
for some rational $e$, and 
some small enough $\varepsilon >0$. 

We will only refer to $\mcd^{(e)}_{\gamma_*,\varepsilon}(\lambda;\eta)$ such that the canyon degree 
of $\cC(\gamma_*)$  is $d >1$.  If two polars are in the same canyon, then their associated disks coincide by definition.

 Let $D_{f}$ be some disk of the union $\mcd^{(e)}_{\gamma_*,\varepsilon}(\lambda;\eta)$.   
%

By \emph{canyon disk} (cf. \cite[\S 5]{PT}) we shall mean in the following such a disk $D_{f}$.

\subsection{Proof of Lemma \ref{l:yY}.}
The bi-Lipschitz homeomorphism
 $\varphi$ maps the Milnor fibre $\{f=\lambda\}\cap B(0;\eta)$ into the Milnor fibre $\{g=\lambda\}\cap B(0;\eta')$, for convenient choices of the Milnor ball radii $\eta$ and $\eta'$. 
Let $D_f$ be a canyon disk associated with canyon $\GC(\gamma_*)$. Then, by \cite[Theorem 5.8]{PT}, there exists a polar arc $\gamma'$ of $g$ such that $\varphi(D_f)=D_g$, where $D_g$ is a canyon disk associated with canyon $\GC(\gamma'_*)$. Making $\lambda \to 0$, we get that $\varphi(J^d(\gamma(y)),y)\in \Horn^{(d)}(\gamma'_*;\varepsilon;\eta),$ 
for some small enough $\varepsilon$ and $\eta$.
This means that we have:
$$ \bigl| \varphi_1\bigl( \gamma(y),y \bigr) -\gamma'(Y) \bigr|= \bigl\| \varphi \bigl(\gamma(y),y\bigr) - \bigl(\gamma'(Y), Y\bigr) \bigr\| =O(| Y^{d}|),
$$
where $Y:= \varphi_2\bigl(\gamma(y),y\bigr)$. 
\fin

\bigskip

\begin{lemma}\cite{HP}, \cite{PT}. \label{l:theconstant}  \ \\
Let $f,g:(\bC^2,0)\to (\bC,0)$ be analytic function germ such that $f=g\circ \varphi$, where $\varphi=(\varphi_1,\varphi_2):(\bC^2,0)\to (\bC^2,0)$ is a bi-Lipschitz homeomorphism. Let $\gamma$ be some polar arc of $f$ tangent to a line $\ell  \in \Cone_0(f)$.  

Then we have $|\varphi_2(\gamma(y),y)|\sim |y|$.  
Moreover, if $d_\gamma >1$, then for all polars $\bar\gamma$ of $f$ such that $\ord\bigl(\gamma(y)-\bar\gamma(y)\bigr)>1$ we have the equality: 
$$\varphi_2\bigl(\bar\gamma(y),y\bigr)=cy+o(y)$$
 with the same constant $c\in \bC^*$.
\end{lemma}
\begin{proof}
We may choose the coordinates in $\bC^2$ such that both $f$ and $g$ are mini-regular in $x$, i.e., that the tangent cones of $f$ and $g$ do not contain the direction $[1:0]$, see the definition in \S\ref{s:prelim}. By our assumptions, the polar $\gamma$ is tangential, i.e., its tangent is included in $\Cone_0(f)$.  This means that $\gamma$ has contact order $k\geq1$ with some root $\zeta$ of $\{f=0\}$, in other words:  
 $$\|(\zeta(y),y)-(\gamma(y),y)\|\sim |y|^k.$$
By the bi-Lipschitz property,  there are constants $0< m < K$ such that in the neighbourhood of the origin we have:
\begin{equation}\label{eq:bounds}
m\|(\zeta(y),y)-(\gamma(y),y)\|\leq \|\varphi(\zeta(y),y)-\varphi(\gamma(y),y)\|\leq K\|(\zeta(y),y)-(\gamma(y),y)\|.
\end{equation}
 
  Since $\varphi$ is  bi-Lipschitz, we have: 
\begin{equation}\label{lem_2prop}
\|\varphi(\zeta(y),y)\|\sim \|(\zeta(y),y)\|\sim |y|.
\end{equation}
We write:
$$
\| (\varphi_1(\zeta(y),y),\varphi_2(\zeta(y),y)) \|=|\varphi_2(\zeta(y),y)| \left\Vert \left(\frac{\varphi_1(\zeta(y),y)}{\varphi_2(\zeta(y),y)},1 \right) \right\Vert.
$$
Since $f=g\circ \varphi$, the root $\zeta$ is sent by $\varphi$ to some root $\eta=(\eta_1,\eta_2)$ of $g$, which means that we have the equality of directions $\left[\frac{\varphi_1(\zeta(y),y)}{\varphi_2(\zeta(y),y)}:1 \right] = \left[\frac{\eta_1}{\eta_2}:1 \right]$. The later tends to the direction of the tangent line $\eta$, which is different from $[1:0]$ by our assumption. Hence this is of the form $[a,1]$, where $a\in \bC$. Consequently:
$$
|\varphi_1(\zeta(y),y)|\leq M|\varphi_2(\zeta(y),y)|,  \ \text{ for some } M>0
$$

By using \eqref{lem_2prop}, we then get:
$$
m|y|\leq |\varphi_2(\zeta(y),y)|+|\varphi_1(\zeta(y),y)|\leq (1+M)|\varphi_2(\zeta(y),y)|,
$$


and  further, by using \eqref{eq:bounds}, this implies:
 
\begin{equation}\label{eq:simy}
 |\varphi_2(\gamma(y),y)|\sim |y| .
\end{equation}

 Note that property \eqref{eq:simy} is independent of the choice of the polar 
 $\gamma$.

By \S \ref{ss:degree} and \S \ref{ss:order-h}, we have the expansion:
\begin{equation}\label{eq:fh}
  f(\gamma(y),y) =ay^h+\cdots + \alpha y^{h+d-1} +\hot,
\end{equation}
where $d=d_\gamma$ is the degree of the gradient canyon $\GC(\gamma_{*})$, and where all the terms before $\alpha y^{d+h-1}$ depend only on the canyon $\GC(\gamma_{*})$.
 
 Lemma \ref{l:yY} provides a polar $\gamma'$ of $g$ such that, for $Y:= \varphi_2\bigl(\gamma(y),y\bigr)$, we also have:
 $$(g\circ \varphi)(\gamma(y),y) = g\bigl(\gamma'(Y)+O(Y^{d}), Y\bigr)
  =a'Y^H+\cdots +O(Y^{H+d-1} ),$$
 where all the terms before $O(Y^{d+H-1})$ depend only on the canyon $\GC(\gamma'_*)$.

We obtain the equality:
\begin{equation}\label{eq:expansions}
 ay^h+\cdots + O(y^{h+d-1} )= a'Y^H+\cdots +O(Y^{H+d-1}).
\end{equation}\label{limit}
 In particular we obtain that  $h=H$.  When $d>1$,  we have $h+d-1 >h$, and therefore we also get:
 \begin{equation}\label{eq:limit}
 Y/y\to c\neq 0,  \text{ when } y\to 0, 
\end{equation}
thus we get  $Y=cy+o(y)$.

Suppose now that $\bar\gamma$ is another polar arc of $f$, tangent to the line $\ell \in \Cone_0(f)$, such that $\GC(\gamma_*)\neq \GC(\bar\gamma_*)$, and let $\bar Y :=\varphi_2(\bar\gamma(y),y)$. 
Using the bi-Lipschitz property of $\varphi$, we have:
$$
|Y- \bar Y|=\|\varphi_2(\gamma(y),y)-\varphi_2(\bar\gamma(y),y)\| \leq \|(\gamma(y),y)-(\bar\gamma(y),y)\|.
$$
Since $\gamma$ and $\bar\gamma$ satisfy
$\ord \|(\gamma(y),y)-(\bar\gamma(y),y)\| > 1$,  this implies that  $Y- \bar Y=o(y)$.  
 \end{proof}

 \subsection{First level bi-Lipschitz invariant}\label{ss:firstlevel}
 
 The  main theorem of \cite{HP} gives the first level continuous Lipschitz invariant. It is a direct consequence of  the above lemma, as an application of \eqref{eq:limit}. Since this will be used in the following,  we state it here (in an equivalent version),  and give a short proof:
 
\begin{corollary}[\cite{HP} main theorem, equivalent version]\label{c:HPmain}\ \\
Let $f,g :(\bC^2,0)\to (\bC,0)$ be holomorphic functions such that $f=g\circ \varphi$, where $\varphi:(\bC^2,0)\to (\bC^2,0)$ is a bi-Lipschitz homeomorphism and $f$ is mini-regular in $x$. Let $\gamma$ be some polar arc of $f$ tangent to a line $\ell  \in \Cone_0(f)$, and let $a$, $h$  be the coefficient and the order of the leading term of $f(\gamma(y),y)$. 

Then the effect  of the bi-Lipschitz map $\varphi$ on $(h,a)$ is the identity on $h$, and the multiplication of $a$ by  $c^{h}$, where $c$ is the non-zero constant provided by Lemma \ref{l:theconstant}.

\end{corollary}
\begin{proof}
Let $d>1$ denote the degree of the canyon $\GC(\gamma_*)$.
 By \S\ref{ss:degree}  and \S\ref{ss:order-h}, we have the following:
\begin{equation}\label{eq:3.0}
\begin{aligned}
f(\gamma(y),y)&=ay^{h}+\cdots +O(y^{h+d-1}),\\
\end{aligned}
\end{equation}
where the terms of orders less than $h+d-1$ are identical for every arc in $\GC(\gamma_{*})$ . Similarly, after applying $\varphi$, by Lemma \ref{l:theconstant}, we obtain:
\begin{equation}\label{eq:4.0}
\begin{aligned}
(g\circ \varphi)(\gamma(y),y)&=bY^{h}+\cdots +O(Y^{h+d-1}),\\
\end{aligned}
\end{equation}
where,  by Lemma \ref{l:theconstant}, $Y$ is a  local variable, and the terms of orders less than $h+d-1$ are the same for every arc in $\GC(\gamma'_{*})$. 

Since $f=g \circ \varphi$, by comparing \eqref{eq:3.0} to \eqref{eq:4.0}, we get:
\begin{equation}\label{eq:5.0}
\begin{aligned}
a y^{h}+\cdots +O(y^{h+d-1})&=b Y^{h}+\cdots +O(Y^{h+d-1}),\\
\end{aligned}
\end{equation}
and by applying the substitution  \eqref{eq:1} provided by Lemma \ref{l:theconstant}, we obtain the equality:
\begin{equation}\label{eq:coeff}
  b=ac^{-h},
\end{equation}
 where $c$ is the same constant for all canyons in $G_\ell(f)$, as established in Lemma \ref{l:theconstant}.  
\end{proof}
 
\

Our next result refines in more depth the equality in Lemma \ref{l:theconstant},  up to a level depending on the involved canyons.  It plays a crucial role in Example \ref{example2}.

\begin{theorem}\label{t:coeff}
Let $d\in \bQ_{+}$ be the degree of the canyon $\GC(\gamma_{*})$.

 For any polar arc $\gamma' \in \GC(\gamma_{*})$, we have:
 \[ \varphi_{2}(\gamma'(y),y) = P(y) + \phi(y)
 \]
where  $P(y) = cy + \cdots$  is a certain finite sum of monomials with rational exponents less than $d-1$, 
 which is well determined by $f$ and $g$.
\end{theorem}

 %
\begin{proof}
By Lemma \ref{l:theconstant} we have:
$$
Y= \varphi_2\bigl(\gamma(y),y\bigr) = cy+ \xi(y).$$

We claim that $\xi(y) = r_{1}y^{\beta_{1}} +o(y^{\beta_{1}})$, where $r_{1}\in \bC$ and $\beta_{1}\in \bQ_{+}$, $ \beta_{1} > 1$, are well determined
by the coefficients of the expansions of $f$ and $g$ displayed in \eqref{eq:expansions}, which are invariant in the respective canyons. 

\

Let $h+q$ and $h+q'$ denote the second exponent in the expansion at the left hand side, and at the right hand side of \eqref{eq:expansions}, respectively, where $q,q' \in \bQ_{+}$ and $q,q'<d-1$.  Then \eqref{eq:expansions} reads:
%
\begin{equation}\label{proof:1}
 ay^h+by^{h+q}+ \cdots + O(y^{h+d-1}) = a'Y^h+b'Y^{h+q'}+\cdots +O(Y^{h+d-1}) 
\end{equation}
 By Lemma \ref{l:theconstant} we have $a'=a\frac{1}{c^h}$, in particular $Y-cy=\xi(y)=o(y)$. 
 Then, by replacing $Y$ in \eqref{proof:1}, we get:
\begin{equation}\label{eq:coeffandexpo}
 ay^h+by^{h+q}+o(y^{h+q}) =a\frac{1}{c^h}( cy+\xi(y))^h+b'( cy+\xi(y))^{h+q'}+o(y^{h+q'}),
\end{equation}

One develops:
$$\left( cy+\xi(y)  \right)^h = c^hy^h + c^{h-1}hy^{h-1}\xi(y)+ \hot, $$
and thus the right hand side of \eqref{eq:coeffandexpo} becomes: 
$$ay^h + a\frac{1}{c}hy^{h-1}\xi(y) + b'c^{h+q'}y^{h+q'}   +o(y^{h+q'}) $$
After reducing the first terms $ay^h$, the equality \eqref{eq:coeffandexpo} becomes:
\begin{equation}\label{proof:2}
by^{h+q}+o(y^{h+q}) =  a\frac{1}{c}hy^{h-1}\xi(y)+ o\bigl(y^{h-1}\xi(y) \bigr)  + b'c^{h+q'}y^{h+q'}   +o(y^{h+q'})
\end{equation}
The study of this equality falls into 3 cases, as follows.\\

\noindent \textbf{Case 1:} $q< q'$.  Then the limit $\lim_{y\to 0} \frac{\xi(y)}{y^{q+1}}$ is well defined, and we deduce:
 $$\xi(y)=\frac{bc}{ah}y^{q+1}+ o(y^{q+1}).$$ 

\

\noindent \textbf{Case 2:} $ q=q'$. Then the limit
$
\lim_{y\to 0}\frac{\xi(y)}{y^{q+1}} = \frac{c(b-b'c^{h+q})}{ah}$
is well defined, and therefore: 
$$\xi(y)=\frac{c(b-b'c^{h+q})}{ah}y^{q+1}+o(y^{q+1}).$$

\noindent \textbf{Case 3:} $q>q'$. 
Then  $a\frac{1}{c}h\lim_{y\to 0}\frac{\xi(y)}{y^{q'+1}}+ b'c^{h+q'}=0$, and therefore
$$\xi(y)=\frac{-b'c^{h+q'+1}}{ah}y^{q'+1}+ o(y^{q'+1})$$

This ends the proof of our claim on the function $\xi(y)$.
Let us call ``Step 1'' this determination of its first term $r_{1}y^{\beta_{1}}$ of $\xi(y)$.   Next, we set:
$$
Y=(cy+ r_{1}y^{\beta_{1}}) + \xi'(y)
$$ 
and we claim that we have $\xi'(y) =  r_{2}y^{\beta_{2}}  + o(y^{\beta_{2}})$, where  $r_{2}\in \bC$,  $\beta_{2}\in \bQ_{+}$, and $\beta_{2} > \beta_{1}$. 

Like in Step 1, we then do a similar case-by-case study of the equality \eqref{proof:1} in which we replace $Y$. This will involve the first two terms of both sides. The determination of the first two terms will be called ``Step 2'', and we continue with another step.
The last step $k$ is defined by the last time when the highest exponent $\beta_{k}$ may influence the terms $<d+h-1$ of the expansions \eqref{proof:1}. This happens when, in the term $c^{h}y^{h}r_{k}\beta_{k}$ of the expansion of $Y^{h}$, the exponent of $y$ is less than $h+d-1$, and therefore we get $\beta_{k}< d-1$.

Altogether, we have shown how to determine the polynomial expression $P(y)$ with rational coefficients. This finishes the proof of our theorem.
\end{proof}

\medskip


Let $\gamma_1$ and $\gamma_2$ be polar arcs of $f$ such that $\GC(\gamma_{1*})\neq \GC(\gamma_{2*})$ and that $\GC(\gamma_{1*}),\GC(\gamma_{2*})\in G_\ell(f)$.  By Lemma \ref{l:yY}, there are polars $\gamma'_1$, $\gamma'_2$ of $g$ such that their canyons $\GC(\gamma'_{1*})$, $\GC(\gamma'_{2*})$ have the same degrees $d_1$, $d_2$ as $\GC(\gamma_{1*})$, $\GC(\gamma_{2*})$, respectively, and such that:

$$
\varphi(\gamma_1(y),y)=(\varphi_1(\gamma_1(y),y),\varphi_2(\gamma_1(y),y))=(\gamma'_1(Y_1)+O(Y_1^{d_1}),  Y_1)
$$
and
$$
\varphi(\gamma_2(y),y)=(\varphi_1(\gamma_2(y),y),\varphi_2(\gamma_2(y),y))=(\gamma'_2(Y_2)+O(Y_2^{d_2}),  Y_2)
$$
where $Y_1$, $Y_2$ are local variables.

\medskip
%
%

By Lemma \ref{l:theconstant} we have:
\begin{equation}\label{eq:1}
Y_i=cy+o(y), \quad i=1,2,
\end{equation}
where $c$ is the same constant for all canyons in $G_\ell(f)$.

\medskip
We are now prepared to introduce a new class of bi-Lipschitz invariants associated with pairs of gradient canyons. 

\subsection{Higher invariants: second level}\label{ss:1st level}

We assume that there are at least two distinct gradient canyons contained in $G_\ell(f)$, $\GC(\gamma_{1*})$ and $\GC(\gamma_{2*})$, of canyon degrees $d_1$ and $d_2$, respectively. If $\delta \ge 1$ denotes the contact order between them, then this must be lower than both, i.e.:
\begin{equation}\label{eq:lower}
 \delta < \min\{d_1, d_2\}.
\end{equation}
 By \cite{PT}, the corresponding canyons of $g$, denoted by  $\GC(\gamma'_{1*})$ and $\GC(\gamma'_{2*})$, have the same degrees and the same contact order $\delta$, cf \S\ref{ss:degree} and \S\ref{ss:contact}.

In order to state the theorem, we now assume that the orders are equal: 
$$\ord f(\gamma_1(y),y) = \ord f(\gamma_2(y),y),$$
 and we  introduce the following notation:

\begin{itemize}
\item $a_1$ is the coefficient of $y^{h}$ in the expansion of $f(\gamma_1(y),y)$,
\item $a_2$ is the coefficient of $y^{h}$ in the expansion of $f(\gamma_2(y),y)$, 
\item $H=\ord \left[\frac{1}{a_1}f(\gamma_1(y),y)-\frac{1}{a_2}f(\gamma_2(y),y) \right]$,
\item $\tilde{a}_1$ is the coefficient of $y^{H}$ in the expansion of $\frac{1}{a_1}f(\gamma_1(y),y)$,
\item $\tilde{a}_2$ is the coefficient of $y^{H}$ in the expansion of $\frac{1}{a_2}f(\gamma_2(y),y)$.
\end{itemize}

With these notations, we have\footnote{After the publication of this manuscript as an arXiv preprint,  N. Nguyen pointed up to us his manuscript \cite{Ng} containing a statement which seems equivalent to our Theorem \ref{c:invar1}.}:
\begin{corollary}\label{c:invar1}
Let $f,g :(\bC^2,0)\to (\bC,0)$ be holomorphic functions such that $f=g\circ \varphi$, where $\varphi:(\bC^2,0)\to (\bC^2,0)$ is a bi-Lipschitz homeomorphism and $f$ is mini-regular in $x$. Let $h=\ord_y f(\gamma_1(y),y) = \ord_y f(\gamma_2(y),y)$.
If 
 \begin{equation}\label{eq:H-cond}
  H<  h+\delta-1,
\end{equation}
 then the effect  of the bi-Lipschitz map $\varphi$ on the pair $\bigl( H, (\tilde{a}_1-\tilde{a}_2)\bigr)$ is the identity on $H$, and  $(\tilde{a}_1-\tilde{a}_2)$ multiplies by $c^{h-H}$.
\end{corollary}
\begin{proof}
By \S\ref{ss:degree} and \S\ref{ss:order-h}, we have the following expansions:
\begin{equation}\label{eq:3}
\begin{aligned}
f(\gamma_1(y),y)&=a_1y^{h}+\cdots +O(y^{h+d_1-1}),\\
f(\gamma_2(y),y)&=a_2y^{h}+\cdots +O(y^{h+d_2-1}),\\
\end{aligned}
\end{equation}
where the terms of orders less than $h+d_1-1$ and $h+d_2-1$ are identical for every arc in $\GC(\gamma_{1*})$ and $\GC(\gamma_{2*})$, respectively. Similarly, after applying $\varphi$, by Lemma \ref{l:theconstant}, we obtain:
\begin{equation}\label{eq:4}
\begin{aligned}
(g\circ \varphi)(\gamma_1(y),y)&=g(\gamma'_1(Y_1),Y_1)+O(Y_1^{h+d_1-1})&=b_1Y_1^{h}+\cdots +O(Y_1^{h+d_1-1}),\\
(g\circ \varphi)(\gamma_2(y),y)&=g(\gamma'_2(Y_2),Y_2)+O(Y_2^{h+d_2-1})&=b_2Y_2^{h}+\cdots +O(Y_2^{h+d_2-1}),\\
\end{aligned}
\end{equation}
where $Y_1,Y_2$ are local variables, cf \eqref{eq:1}, and the terms of orders less than $h+d_1-1$ and $h+d_2-1$ are the same for every arc in $\GC(\gamma'_{1*})$ and $\GC(\gamma'_{2*})$, respectively. Since $f=g \circ \varphi$, by comparing \eqref{eq:3} to \eqref{eq:4}, we get:
\begin{equation}\label{eq:5}
\begin{aligned}
a_1 y^{h}+ \cdots + a_1\tilde{a}_1y^H   + \cdots +  O(y^{h+d_1-1})&=b_1 Y_1^{h}+ \cdots + \overline{b}_1Y_1^H   + \cdots +O(Y_1^{h+d_1-1}),\\
a_2 y^{h}+\cdots +  a_2\tilde{a}_2y^H   + \cdots + O(y^{h+d_2-1})&=b_2 Y_2^{h}+\cdots +\overline{b}_2Y_2^H   + \cdots +O(Y_2^{h+d_2-1}).\\
\end{aligned}
\end{equation}

\

As $\varphi$ is bi-Lipschitz, we have:
\begin{equation}\label{eq:delta}
|Y_1-Y_2|\leq\|\varphi(\gamma_1(y),y)-\varphi(\gamma_{2}(y),y)\| \sim \|(\gamma_1(y),y)-(\gamma_{2}(y),y)\|,
\end{equation}
and since $\delta=\ord_y(\gamma_1(y)-\gamma_{2}(y)) \ge 1$, we then obtain:
 $$Y_2=Y_1+O(Y_1^{\delta}).$$ 
We now use the assumption \eqref{eq:H-cond}. 
By using this substitution, together with the inequality \eqref{eq:lower},  and after dividing by the leading coefficients $a_i$, we may rewrite \eqref{eq:5} as follows:
\begin{equation}\label{eq:6}
\begin{aligned}
y^{h}+\cdots +\tilde{a}_1y^H +\cdots +O(y^{h+\delta-1})&=\frac{b_1}{a_1} Y_1^{h}+\cdots + \frac{\overline{b}_1}{a_1}Y_1^H + \cdots +O(Y_1^{h+ \delta -1}),\\
y^{h}+\cdots +\tilde{a}_2y^H +\cdots +O(y^{h+\delta-1})&=\frac{b_2}{a_2} Y_1^{h}+\cdots + \frac{\overline{b}_2}{a_2}Y_1^H + \cdots +O(Y_1^{h+\delta-1}).\\
\end{aligned}
\end{equation}

By applying \eqref{eq:1}, we get  $\frac{b_1}{a_1}=\frac{b_2}{a_2}= c^{-h}$, where $c$ is a constant for all canyons in $G_\ell(f)$, as established in Lemma \ref{l:theconstant}.

\

Now, by subtracting the second equation from the first, we obtain:

\begin{equation}\label{eq:7}
\begin{aligned}
\left(\tilde{a}_1-\tilde{a}_2\right) y^{H}+o(y^{H})=&  \left(\frac{\overline{b}_1}{a_1}-\frac{\overline{b}_2}{a_2}\right) Y_1^{H}+o(Y_1^{H})
\end{aligned}
\end{equation}
where $\left(\tilde{a}_1-\tilde{a}_2\right) \not= 0$ due to our assumption that $H> h$ .

By using \eqref{eq:1},  we conclude that  we must have the equality:
$$
\left(\tilde{a}_1-\tilde{a}_2\right)=c^{H}  \left(\frac{\overline{b}_1}{a_1}-\frac{\overline{b}_2}{a_2}\right) = c^{H}c^{-h}  \left(\frac{\overline{b}_1}{b_1}-\frac{\overline{b}_2}{b_2}\right) = c^{H-h}  \left(\tilde{b}_1-\tilde{b}_2\right),
$$
where 
\begin{itemize}
\item $\tilde{b}_1$ denotes the coefficient of $y^{H}$ in the expansion of $\frac{1}{b_1}(g\circ \varphi)(\gamma_1(y),y)$,
\item $\tilde{b}_2$ denotes the coefficient of $y^{H}$ in the expansion of $\frac{1}{b_2}(g\circ \varphi)(\gamma_2(y),y)$.
\end{itemize}

Conclusion: if the above condition $H<  h+\delta-1$ holds, then the action of the bi-Lipschitz map $\varphi$
on the pair $\bigl(H, (\tilde{a}_1-\tilde{a}_2)\bigr)$ is the identity on the first component $H$, and the multiplication of $(\tilde{a}_1-\tilde{a}_2)$  by $c^{h-H}$. 
\end{proof}


\subsection{Higher invariants: third level}\label{ss:higherlevel}\ \\
Let us give the reccursive receipt for constructing higher  level bi-Lipschitz invariants.
 We only consider here the third step, since the next ones are analogous. Each step hold upon a condition
similar to \eqref{eq:H-cond}, and the length of the chain of such higher invariant is limited by the size of the contact between canyons.
 
 In the third step we need at least 3 pairwise disjoint  gradient canyons $\GC(\gamma_{1*})$, $\GC(\gamma_{2*})$, $\GC(\gamma_{3*})$.
  

Let us assume the existence of three polar arcs $\gamma_i \in \GC(\gamma_{i*})$, such  that we have the following coincidences of orders:
\begin{equation}\label{eq:h3pol}
h:=\ord f(\gamma_1(y),y)=\ord f(\gamma_2(y),y)=\ord f(\gamma_3(y),y),
\end{equation}
and
\begin{equation}\label{eq:Hdiff}
 H:=\ord  \biggl[\frac1a_1f(\gamma_1(y),y)- \frac1a_2f(\gamma_2(y),y)\biggr]=\ord \biggl[\frac1a_3f(\gamma_3(y),y)- \frac1a_1f(\gamma_1(y),y)\biggr].
\end{equation}
Let $\delta$ denote the contact order between $\GC(\gamma_{1*})$ and $\GC(\gamma_{2*})$, and let $\delta'$ denote the contact order between $\GC(\gamma_{1*})$ and $\GC(\gamma_{3*})$.

We divide each of the two above differences in \eqref{eq:Hdiff} by $(\tilde{a}_1-\tilde{a}_2)$, and $(\tilde{a}_3-\tilde{a}_1)$, respectively, where $\tilde{a}_i$ denotes the coefficient of $y^{H}$ in the expansion of $\frac{1}{a_i}f(\gamma_1(y),y)$.

The resulting expansions in $y$ will then coincide before the term with a certain exponent $H' >H$. We also need that $H'$ satisfies the condition:
$$H' < \min\{h+\delta-1, h+\delta'-1\}.$$
From this moment on we do exactly like in the proof of Thm 2.5, i.e., by subtracting one from the other. Denoting by $A_{ij}$ the coefficient of $y^{H'}$ in the expansion of $$(\tilde{a}_i-\tilde{a}_j)^{-1}\biggl[\frac1a_i f(\gamma_i(y),y)- \frac1a_j f(\gamma_j(y),y)\biggr],$$ for $(i,j)\in\{(1,2), (3,1)\}$, we have thus proved the following statement:

\begin{corollary}\label{c:invar2} Let $f,g :(\bC^2,0)\to (\bC,0)$ be holomorphic functions such that $f=g\circ \varphi$, where $\varphi:(\bC^2,0)\to (\bC^2,0)$ is a bi-Lipschitz homeomorphism and $f$ is mini-regular in $x$. Assume that the conditions \eqref{eq:h3pol} and \eqref{eq:Hdiff} hold.

If $H'< \min\{h+\delta-1, h+\delta'-1\}$,  then the effect of the bi-Lipschitz map $\varphi$ on $\bigl( H',  ( A_{12} - A_{31}) \bigr)$ is the identity on $H'$, and the multiplication of $(A_{12} - A_{31} )$ by  $c^{H-H'}$, where $c$  is the non-zero constant provided by \eqref{eq:1}.
\fin
\end{corollary}

\medskip


\section{Examples}\label{s:examples}

Computing the invariants of Corollary  \ref{c:invar1} in the family of functions $F_t$ of Example \ref{example1} displayed below shows that, for small enough generic parameters $t$ and $s$, $F_t$ and $F_s$ are not bi-Lipschitz equivalent.  
 In contrast, the invariants of Corollary  \ref{c:invar1} turn out to be not sufficient to settle the same question for the family $G_t$ in Example \ref{example2}.
It is by the key application of Theorem \ref{t:coeff} that we can decide that $G_t$ and $G_s$ are generically not bi-Lipschitz equivalent.


\begin{example}\label{example1}
Let us consider the family of function germs $F_t(x,y)=\frac{1}{3}x^3-t^2xy^{10}+y^{12}$, for $t$ in a small neighbourhood of $0\in \bC$. 

It appears that $F_t$ has two polar arcs: $\gamma_1(y)=ty^{5}$ and $\gamma_2(y)=-ty^5$. 
By direct computations we get the canyon degrees $d_{\gamma_1}=d_{\gamma_2}=6$, and their contact order $\delta :=\ord_y(\gamma_1(y)-\gamma_2(y))=5$. This also shows that the two polars belong to different canyons, i.e. $\GC(\gamma_{1*})\neq \GC(\gamma_{2*})$, with $\GC(\gamma_{1*}), \GC(\gamma_{2*})\subset G_{\ell}(F_t)$, where $\ell=\{x=0\}$. 
 We also get:
\begin{equation}\label{eq:F_t}
 F_t(\gamma_1(y),y)=y^{12}-\frac{2}{3}t^3y^{15},\quad  F_t(\gamma_2(y),y)=y^{12}+\frac{2}{3}t^3y^{15},
\end{equation}
thus $h=12$, and the leading coefficients of $F_t(\gamma_1(y),y)$, $F_t(\gamma_2(y),y)$ are $a_1=1$, $a_2=1$.

Comparing to \eqref{eq:F_t}, we then get:
$$
H:=\ord\left[a_1^{-1}F_t(\gamma_1(y),y)-a_2^{-1}F_t(\gamma_{2}(y),y) \right]=15.
$$

\medskip

We will follow the proof of Corollary  \ref{c:invar1} in order to contradict  the assumption that there exists a bi-Lipschitz homeomorphism $\varphi$ such that  $F_{t} = F_{s}\circ \varphi$ for some parameters $s,t\in \bC$ close enough to $0$.

 By Lemma \ref{l:theconstant} we have: 
\begin{equation}\label{eq:varphi-old}
 \varphi(\gamma_1(y),y) =(\gamma'_1(Y_1)+O(Y_1)^{d} ,Y_1)  \mbox{ and  } \varphi(\gamma_2(y),y) =(\gamma'_2(Y_2)+O(Y_2^{d}) ,Y_2),
\end{equation}
where $d:=d_{\gamma_1}=d_{\gamma_2}$, and where $Y_1, Y_2$ are local variables, and  $\gamma'_1$, $\gamma'_2$ are corresponding  polar arcs of $F_s$.
Since we have two options for the correspondence  by $\varphi$ between the pair of canyons of $F_s$ and of $F_t$, the computation of the invariants of  Corollary  \ref{c:invar1} falls into two cases.

\medskip

\noindent \textbf{Case 1.} 
  $\gamma'_1(Y_1)=sY_1^5$ and  $\gamma'_2(Y_2)=-sY_2^5$.

Then by our assumption we have: $F_t(\gamma_i(y),y) = F_s(\varphi(\gamma_i(y),y))=F_s(\gamma'_i(Y_i),Y_i) $, $i=1,2$,
and by applying \eqref{eq:varphi-old} and \eqref{eq:5}, we obtain the following equalities: 
\begin{equation}\label{ex2:1}
\begin{aligned}
y^{12}-\frac{2}{3}t^3y^{15}&=Y_1^{12}-\frac{2}{3}s^3Y_1^{15},\\
y^{12}+\frac{2}{3}t^3y^{15}&=Y_2^{12}+\frac{2}{3}s^3Y_2^{15}.
\end{aligned}
\end{equation}
In particular, the leading coefficients  $b_1$ and $b_2$ of $F_s(\gamma'_1(Y_1),Y_1)$ and $F_s(\gamma'_2(Y_2),Y_2)$, respectively,  are both equal to 1.  At this moment, since  $\frac11= \frac{ b_1}{a_1}$,
we deduce that $c^{-12}=1$.

On the other hand, since $\ord_y\bigl(\gamma_1(y)-\gamma_{2}(y)\bigr) = 5$, by  \eqref{eq:delta} we get:
\begin{equation}\label{ex2:change}
Y_2=Y_1+O(Y_1^{5}) = Y_1+pY_1^{5}+ o(Y_1^5),
\end{equation}
for some $p\in \bC$.  
Applying the change of variables \eqref{ex2:change} in the equations of \eqref{ex2:1}, with signs ``$-$'' and ``$+$", respectively,  yields:
\begin{equation}\label{eq:pm}
 \begin{aligned}
y^{12}\pm \frac{2}{3}t^3y^{15}&=\bigl( Y_1+pY_1^{5}+o(Y_1^5) \bigr)^{12} \pm \frac{2}{3}s^3\bigl( Y_1+pY_1^{5}+ o(Y_1^5)\bigr) ^{15} \\
\ &=Y_1^{12}\pm \frac{2}{3}s^3Y_1^{15}+12pY_1^{16} +o(Y_1^{16}).
\end{aligned}
\end{equation}
By subtracting one equation from the other, we get:
\begin{equation}\label{ex2:diff}
-\frac{4}{3}t^3y^{15}=-\frac{4}{3}s^3Y_1^{15} + o(Y_1^{15}).
\end{equation}

Referring to \eqref{eq:7}, we have in our case, since $H=15$:
$$
\tilde{a}_1-\tilde{a}_2=-\frac{4}{3}t^3 \quad \tilde{b}_1-\tilde{b}_2=-\frac{4}{3}s^3.
$$
where $\tilde{a}_1 = -\frac{2}{3}t^3 $ and $\tilde{a}_2 = \frac{2}{3}t^3$ are the coefficients of $y^H = y^{15}$ in the expansions of $a_1^{-1}F_t(\gamma_1(y),y)$ and   $a_2^{-1}F_t(\gamma_{2}(y),y)$, respectively, while $\tilde{b}_1 = -\frac{2}{3}s^3 $ and $\tilde{b}_2 = \frac{2}{3}s^3$ are the corresponding coefficients of $Y_1^{15}$ at the right hand side of \eqref{eq:pm}.

\sloppy

By Corollary  \ref{c:invar1}, we must  have $\frac{\tilde{b}_1-\tilde{b}_2}{\tilde{a}_1-\tilde{a}_2}=c^{h-H}$, which in our case  becomes $\frac{s^3}{t^3}=c^{-3}$, while we have seen before that $c^{-12}=1$. The two equalities do not have a common solution $c$ . This contradiction finishes our proof
 that, in Case 1,
  $F_t$ is not bi-Lipschitz equivalent to $F_s$ for a generic choice of values $t$ and $s$ of the parameter in our family. 
\bigskip

\noindent \textbf{Case 2.} 
$\gamma'_1(Y_1)=-sY_1^5$ and  $\gamma'_2(Y_2)=sY_2^5$.

The computations are faithfully similar to Case 1, with a difference of sign, for instance instead of \eqref{ex2:diff} we get:
$$-\frac{4}{3}t^3y^{15}=\frac{4}{3}s^3Y_1^{15} + o(Y_1^{15}).$$
At the end we see that the constant $c$ must satisfy two equalities, namely $-\frac{s^3}{t^3} = c^{-3}$ and $c^{-12} = 1$, which can happen only for $t$ and $s$ satisfying a certain relation. The same conclusion as in Case 1 then follows, i.e. that  $F_t$ is not bi-Lipschitz equivalent to $F_s$ 
for  a generic choice of the values $t$ and $s$.
\end{example}

\bigskip


\begin{example}\label{example2}
We consider now the family of function germs $G_t=x^3+y^{12}+xy^9+ty^{13}$, for $t$ in a small neighbourhood of $0\in \bC$. 

 There are two polars, independent of $t$,  namely:
  $$\gamma_a(y)=ay^{\frac{9}{2}} \mbox{ and } \gamma_{-a}(y)=-ay^{\frac{9}{2}},$$
   where $3a^2+1=0$, of contact order  $\delta = \frac{9}{2}$, in different canyons $\GC(\gamma_a)\neq \GC(\gamma_{-a})$, both tangent to the line $\ell =\{ x=0\}$, and of canyon degrees  $d(\gamma_a)=d(\gamma_{-a})=\frac{13}{2}$.     We also find easily $h=12$.

We follow the same  pattern as in the proof and the computations of the preceding example, namely we start by assuming that there is a bi-Lipschitz homeomorphism $\varphi$ such that
 $G_t=G_s\circ \varphi$. 
Here too, there are two  options for corresponding the two canyons of $G_t$ and $G_s$ by the bi-Lipschitz homeomorphism $\varphi$.

\medskip

By Theorem \ref{t:coeff}, we have:
\begin{equation}\label{eq1:varphi}
 Y=y+\frac{t-s}{12}y^2+\hot
\end{equation}

\medskip

 In one of the cases, call it Case 1,  we  get
$
H=\frac{27}{2}, \quad \tilde{a}_1-\tilde{a}_2=\frac{4a}{3}
$, and the condition $H<h+\delta-1=\frac{31}{2}$ holds.

After some more steps in the algorithm displayed in Example \ref{example1},  we get the following  equality similar to \eqref{ex2:diff}:
\begin{equation}\label{diff}
\frac{4a}{3}y^{\frac{27}{2}}=\frac{4a}{3}Y^{\frac{27}{2}}-12pY^{\frac{31}{2}} +\hot
\end{equation}
where $p\in \bC$ comes from the relation $Y_1=Y+pY^{\frac{9}{2}}+ \hot$ similar to \eqref{ex2:change}.

Pursuing the algorithm, we find  (cf Corollary  \ref{c:invar1}) that the effect of the bi-Lipschitz map on the pair $\bigl(H, (\tilde{a}_1-\tilde{a}_2)\bigr)$ is the identity on $H$, and the multiplication of $(\tilde{a}_1-\tilde{a}_2)$ by  $c^{h-H}=c^{-\frac{3}{2}}=1$. Note that \eqref{eq1:varphi} tells that $c=1$.  Up to this point,  we do not get any contradiction, thus no proof or disproof of our assumption. 

\medskip

To solve the dilemma, we need to make more efficient use of  Theorem \ref{t:coeff} and thus of the relation \eqref{eq1:varphi}, as follows. Replacing \eqref{eq1:varphi} into the right-hand side of \eqref{diff}, yields the equality:
$$
\begin{aligned}
\frac{4a}{3}y^{\frac{27}{2}}&=\frac{4a}{3}y^{\frac{27}{2}} + \frac{4a}{3}\cdot \frac{27}{2}\cdot \frac{(t-s)}{12}y^{\frac{29}{2}}-12py^{\frac{31}{2}} +\hot
\end{aligned}
$$
which is impossible if $t\not= s$.   Therefore the conclusion of Case 1 is that $G_t$ is not bi-Lipschitz equivalent to $G_s$ for any choice $s\not= t$.

The Case 2 is faithfully similar, along the same algorithm. It leads to the condition $c^{h-H}=c^{-\frac{3}{2}}=-1$, which has to be compared to  $c=1$ provided by \eqref{eq1:varphi}, thus no solution for $c$.
We may draw the conclusion in this case, and actually the final conclusion of the example, as follows: $G_t$ is not Lipschitz equivalent to $G_s$ for any choice of $s\not= t$.
\end{example}


\end{document}